\documentclass[12pt,a4]{amsart}
\usepackage{amsmath,amssymb,url}
\usepackage{xspace}
\usepackage{hyperref}
\newcommand{\arXiv}[1]{\href{http://arXiv.org/abs/#1}{\texttt{arXiv:#1}}\xspace}
\newcommand{\IR}{\mathbb{R}}

\newcommand{\xc}[1]{}

\newcommand{\IC}{\mathbb{C}}

\newcommand{\IZ}{\mathbb{Z}}
\newcommand{\IF}{\mathbb{F}}
\newcommand{\IA}{\mathbb{A}}

\newcommand{\QLo}{{\sf QL}}
\newcommand{\QL}{{\sf QL}}
\newcommand{\mc}[1]{\mathcal{#1}}
\newcommand{\mol}{\mc{MOL}}
\newcommand{\molf}{{\mc{MOL}_{fd}}}
\newcommand{\lat}{{\sf L}}

\newcommand{\calNP}{\textsf{NP}\xspace}
\newcommand{\calBP}{\mathcal{BP}\xspace}
\newcommand{\PSPACE}{\textsf{PSPACE}\xspace}
\newcommand{\calPSPACE}{\PSPACE}

\theoremstyle{plain}
\newtheorem{thm}{Theorem}

\newtheorem{pro}[thm]{Proposition}

\begin{document}

\title{Modular ortholattices and the ``Third  Life of Quantum Logic''}
\author{Christian Herrmann}
\address[Christian Herrmann]{Technische Universit\"{a}t Darmstadt FB4\\Schlo{\ss}gartenstr. 7\\64289 Darmstadt\\Germany}
\email{herrmann@mathematik.tu-darmstadt.de}


\begin{abstract}
In the editor's introduction to the special volume 
``The third life of quantum logic: Quantum logic
inspired by quantum computing'' of this journal, 
Dunn,  Moss, and  Wang
discussed  the r\^{o}le of modular ortholattices.
The present note is to provide some  background
for the results and problems mentioned there.
In particular, we
recall  the two ``limits'' of subspace ortholattices
of finite dimensional Hilbert spaces,
the latter being structures related directly to quantum
mechanics and quantum  computation.
These  two ``minimal'' continuous geometries have been
constructed  by von Neumann and shown  to be non-isomorphic.
\end{abstract}
\maketitle
\section{Introduction}
For a  short history of quantum logic we refer to
Dunn,  Moss, and  Wang \cite{dunn}.
Semantics in quantum logic 
is based on non-distributive structures such as the
 ortholattices $\lat(H)$ of subspaces of inner product spaces $H$,
the motivating case from quantum mechanics being Hilbert spaces.
The connectives ``and'', ``or'', and ``not'' 
are interpreted as intersection, closure of sum, and orthogonal complement.
In the seminal paper of Birkhoff and von Neumann \cite{BirkhoffVonNeumann},
the spaces $H$ are of finite dimension; for such, $\lat(H)$
satisfies Dedekind's   modular law,
while in infinite dimension only the orthomodular law holds. 
The further development
of quantum logic, that is its ``first two lives'',
focussed on infinite dimension and general orthomodular
structures. It was mainly von Neumann 
who investigated  modularity 
beyond finite dimension: In his
work on continuous geometries related to certain
rings of operators.

Interest in the modular case  was renewed
in
the ``third life'' of quantum logic, inspired by quantum computing:
In \cite{Hagge1}, Dunn, Hagge, Moss, and Wang 
discussed 
quantum logic tautologies (i.e. the
 equational theory) of $\lat(H)$ where 
$H$ is a finite dimensional Hilbert space.
In \cite{dunn}, this discussion was further elaborated and extended beyond
finite dimension.
The purpose of the present note is to 
provide  some background (and, maybe, easier access) to   results 
and problems in \cite{dunn} which concern modular ortholattices.  

In particular, we recall von Neumann's result \cite{neu2} 
that, for the hyperfinite type II$_1$ von Neumann algebra factor $\mc{R}$,
the ortholattice $\lat(\mc{R})$ of projections  
is not isomorphic to von Neumann's well known
example $CG(\IC)$ \cite{neu3} of a continuous geometry,
obtained as the metric completion of a discrete
construction.
This indicates that the approach to $\lat(\mc{R})$,
outlined in
 \cite[p.454]{dunn}, is somehow  problematic
and that exploring
the set $\QL(\lat(\mc{R}))$ of quantum logic tautologies 
 of $\lat(\mc{R})$
has to be done in the framework of  
von Neumann algebras. As shown by Luca Giudici (cf. \cite{hn})
such approach is indeed possible, proving
that $\QL(\lat(\mc{R}))$ is the intersection
of the $\QL(\lat(H))$, $H$ ranging over all finite dimensional Hilbert
spaces. Thus,
in view of \cite{hard}, $\lat(\mc{R})$
and $CG(\IC)$ have the same quantum logic tautologies,
namely those shared by all $\lat(H)$, $H$ ranging over
finite dimensional Hilbert spaces.
Though, more investigations appear 
to be needed concerning the
r\^{o}le of these
two ``limits'' of
finite  dimensionals 
in 
 the logic approach to quantum mechanics and
quantum computation. 

For each $H$, with $3\leq \dim H <\omega$,
as well as for $\lat(\mc{R})$ the set of tautologies is
decidable \cite{hn,hard}; in all cases,
 the complexity of the complementary problem  
is complete for the same class in the Bum-Shub-Smale
model of non-deterministic real computation
\cite{arxiv}.
 Satisfiability is  complete for  this class,
given fixed $H$ \cite{arxiv}, undecidable for the class
of all  $\lat(H)$ as well as for $\lat(\mc{R})$ and
$CG(\IC)$ \cite{sat}.

As observed in \cite{dunn},
current interest focusses on the categorical
approach to quantum computing as in  
the special volume of this journal. Though, the questions 
considered in \cite{dunn}
make sense also for the additive category
of finite dimensional Hilbert spaces enriched with
adjunction. Transfer of results may rely
on the correspondence of $\lat(H)$ respectively $\lat(H^3)$ and
the $*$-regular ring of endomorphisms of $H$.

\section{Geometric background}
A  \emph{modular ortholattice}, shortly MOL,
is a modular lattice $L$ with bounds $0,1$ and an orthocomplementation
$x \mapsto x'$ (cf.
Section 1.2 of \cite{dunn}). $\mol$ denotes the class of all MOLs.
 $L$ is of finite \emph{dimension}
or \emph{height} $d$ if some/any maximal chain in $L$ has $d+1$ elements,
we write $d=d(L)$ and denote by $\molf$
the class of all MOLs of finite dimension.
Also, for $L\in \mol$,
 $d(L)\geq d$ means that $L$ contains  $d+1$-element  chains.

Now, consider a MOL $L$ of $d(L)=d$.
 Up to isomorphism, $L$ is the subspace lattice $\lat(P)$ of a 
$d-1$-dimensional projective space $P$
 with an anisotropic polarity,
providing the involution on $\lat(P)$.
The lattice 
 $L$ is isomorphic to a direct product of  simple lattices
$\lat(P_i)$ where the $P_i$ are the irreducible components of $P$
(cf. \cite{birk}).
Here,  for  $q\not\leq p'$ one has
$p$  perspective to $q$ via $(p +q)p'$
(we write $x+y$ and $xy$ for joins and meets with the usual rules
for omitting brackets). 
Hence, the polarity on $P$ induces one on each $P_i$ and
 the direct product decomposition of  the lattice  $L$ 
into the $\lat(P_i)$ is also one of  ortholattices.
In particular, for $L\in \molf$ the following are
equivalent: $L$ is directly irreducible, $L$ is subdirectly irreducible,
$L$ is simple.

Again, consider $L \in\molf$.
If  $L$ is simple and  $d\geq 4$ then $L$ is Arguesian \cite{jon}
since the associated projective  space is desarguean. 
If $L$ is simple,  Arguesian, and  $d\geq 3$ then 
$L$ is isomorphic  to the  lattice $\lat(V)$ of linear subspaces
of a $d$-dimensional vector space $V$ over a  division ring $F$ with
involution, endowed with an anisotropic 
sesquilinear form which is hermitean 
w.r.t. this involution (and all the latter give
rise to a simple Arguesian MOL).
 This is in essence  Birkhoff and von Neumann
\cite{BirkhoffVonNeumann};
$F$, $V$, and the form  are determined by $L$  up to ``isomorphism''
cf. \cite[Section 14]{Faure}.
$F$ may be quite far away from the complex number field
(cf. \cite[p.449]{dunn}); e.g. $F$ may be the field  extension of a finite field
by 
$d$ algebraically independent elements,  the involution being identity.

.

\section{Universal algebraic background}
Any interval $[b,c]$ of an MOL $L$ is
also an MOL with the induced orthocomplement
$x \mapsto x'c+b$
and isomorphic to the section $[0,a]$ where $a=b'c$.
Any section $[0,a]$ is
 a homomorphic image
of a sub-ortholattice of $L$, namely
of $[0,a] \cup [a',1]$.
Any homomorphic image $L'$  of 
 $L\in \molf$ is isomorphic to the  section $[0,a]$ 
 of $L$, $a$ the smallest preimage of the top element
of $L'$. For an MOL, any congruence relation of the lattice
reduct is also one of the ortholattice $L$.

Within $\mol$, any finite conjunction of identities
is equivalent to a \emph{tautology}, that is an
identity of the form $t=1$. 
The   equational theory of a class $\mathcal{C} \subseteq \mol$
is also addressed as the \emph{Quantum Logic}
$\QLo(\mathcal{C})$ of this class.
The \emph{variety} $\mc{V}(\mc{C})$
generated by $\mc{C}$
is obtained as the homomorphic images of sub-ortholattices
of direct products of members of $\mc{C}$ and 
is the model class of $\QLo(\mathcal{C})$.
By J\'{o}nsson's Lemma,
any subdirectly irreducible member of 
$\mc{V}(\mc{C})$ is a homomorphic image
of a sub-ortholattice of an ultraproduct
of members of $\mc{C}$.

In particular, if there is $d<\omega$ such that
$d(L) \leq d$ for each $L\in \mc{C}$ then
also $d(S) \leq d$ for each subdirectly irreducible 
 $S\in \mc{V}(\mc{C})$. 
For  $d(L)<\omega$ and subdirectly irreducible $S$
 it follows that
$S \in \mc{V}(L)$ if and only if $S$ 
embeds into an ultrapower of  section $[0,a]$ of $L$,
$d([0,a])=d(S)$. Further on, for $L_i\in\molf$, $\QLo(L_1)=\QLo(L_2)$
if and only if $L_1$ and $L_2$ have the same
universal theory, in particular $d(L_1)=d(L_2)$.

\section{Dimension axioms} Cf. \cite[p.450]{dunn}.
Axioms   granting $d(L) < d$ for lattices $L=\lat(V)$
can be obtained by excluding $d$-dimensional 
  ``coordinate systems''  
in intervals of $L$. Bergman and Huhn \cite{berg,huhn}
used $d$-\emph{diamonds},  that is $a_0, \ldots ,a_{d}$
any $d$ of which are independent in an interval
$[a_\bot,a_\top]$ and have join $a_\top$.
Within a modular lattice, if $a_i=a_j$ for some $i\neq j$ then
$a_\bot=a_\top$; that is, the $d$-diamond is \emph{trivial}. 
For a  $d-1$-dimensional projective space $P$,
non-trivial $d$-diamonds in $\lat(P)$
 are exactly the
systems of $d+1$ points any  $d$ of which are in general
 position; such exist if and only $P$ is irreducible.

There are terms $t^d_i(\bar z)$ in variables $\bar z=(z_0,\ldots z_d)$ 
such that, for any substitution  $\bar a$ in a modular
lattice,  the $t^t_i(\bar a)$  form an $d$-diamond 
and such that  $t^d_i(\bar a)=a_i$ if $\bar a$ is a $d$-diamond.
This means that the modular lattice (ortholattice) freely generated by a $d$-diamond,
considered as system of generators and relations,
is a projective modular (ortho-)lattice. The case $d=2$ can
be read off the diagram of the modular lattice with $3$ free
generators.

Slightly modifying the definition of  von Neumann,
a $d$-\emph{frame} is given by  elements
$a_1, \ldots ,a_{d}$, independent in an interval $[a_\bot,a_\top]$
such that $a_\top=\sum_i a_i$, 
and axes of perspectivity from $a_1$ to $a_j$, $j\neq 1$.
Such are systems of generators and relations
equivalent to $d$-diamonds within modular lattices.
Terms in  analogy to the above can be obtained,
easily, by recursion over $d$.

\begin{pro}
There is a sequence $\delta_d(\bar z)$ of
  $d+2$-variable lattice identities
such that $\delta_d(\bar z)$ is valid in the MOL 
 $L$ if and only if $L$ is a subdirect
product of MOLs $L_i$ with $d(L_i)\leq d$.
In particular, $L$ of $d(L)=d$
is simple if and only if $\delta_{d-1}$ is 
not valid in $L$.
\end{pro}

\begin{proof}
The identity  $\delta_d(\bar z)$ can be given in the form 
$\prod_i t_i^{d+1}=\sum_i t_i^{d+1}$.
Such, is valid in $L$ if $d(L)\leq d$, obviously.
On the other hand, assume $\delta_d$
valid in $L$. According to \cite{jon},
the lattice $L$
embeds into a direct product
of lattices $\lat(P_i)$, $P_i$ an irreducible projective space,
 in which $\delta_d$ is valid, too.  Thus,
the $P_i$ have dimension at most $d-1$ and  $L$ 
has all subdirect factors $L_j$ of $d(L_j)\leq d$.
\end{proof}
Particularly simple  
$\delta_d(\bar z)$ are the $d$-distributive laws of
\cite{berg,huhn}. From \cite{baer} it follows that finite 
MOLs are $2$-distributive.
Identities characterizing $d(L)$ for
$L=\lat(H)$, $H$ a finite dimensional
Hilbert space have been established by \cite{Hagge2,dunn},
for simple $L\in \molf$  by \cite{frey}.

\section{Equational theory} Cf. \cite[p.452-3]{dunn}.
For a $*$-subfield $\IF$ of $\IC$ (with conjugation)
consider  $\IF^d$ with the canonical scalar product.
We write $\QLo(\IF^d)=\QLo(\lat(\IF^d))$. 
Clearly, $\QLo(\IF_1^{d_1}) \subseteq \QLo(\IF_2^{d_2})$
if $\IF_2 \subseteq \IF_1$ and $d_2\leq d_1$;
and inequality holds if $d_2<d_1$.

Let $\IA$ denote the $*$-field of algebraic numbers
and recall that $\IA$ and $\IC$ are elementarily equivalent,
and that so are $\IA\cap \IR$ and $\IR$. Thus
$\QLo(\IF^d)=\QLo(\IC^d)$ for all $\IF\supseteq \IA$ and
 $\QLo(\IF^d)=\QLo(\IR^d)$ for all $\IF\subseteq \IR$ with $\IA\cap \IR
\subseteq \IF$. Also, observe that $\QLo(\IC^d) \subseteq
\QLo(\IR^{2d})$.

Let $\mc{N}=\mc{V}\{\lat(\IC^d)\mid d<\omega\}$,
that is $\QLo(\mc{N})= \bigcap_{d< \omega} \QLo(\IC^d)\}$.
\begin{thm}
Let $\IA\cap \IR \subseteq \IF \subseteq \IC\}$. The following hold.
\begin{enumerate}
\item $\QLo(\mc{N})=\QLo(\IF^d) \mid d<\omega)$.
\item $\QLo(\mc{N})=\QLo(L)$ where 
$L$ is  a direct limit of ortholattices
$\lat(\IF^d)$, $d \to \infty$. 
\item
$\QLo(\mc{N})=\QLo(CG(\IF))$,
where $CG(\IF)$ is
von Neumann's   example  \cite{neu3}
of a continuous geometry $CG(\IF)$
obtained
as the metric completion of an ortholattice as in 2. 
\item $\QLo(\mc{N})=\QLo(\lat(\mc{A}))$
where $\lat(\mc{A})$ is  
the projection ortholattice $\lat(A)$
 of some/any finite type  II$_1$ von Neumann algebra factor $\mc{A}$.
\item $\QLo(\mc{N})=\QLo(\mc{C})$,
 $\mc{C}$ the  class of  projection ortholattices
 of   finite Rickart $C^*$-algebras.
\end{enumerate}
\end{thm} 
(1) follows from the above remarks, (2) is obvious,
(3) follows from the fact that the metric completion of $L$
is in $\mc{V}(L)$;  (4)
 is due to Luca Giudici \cite{hn},  (5) 
to \cite{hs}. (4) and (5) 
rely on  an orthogonality preserving embedding  into the
lattice  of all  subspaces of some inner product space --
for certain  countable sub-ortholattices in (4), 
derived form the GNS-construction in (5).

 Recall that, in spite of some structural analogies
(cf. \cite{ara}),
 $CG(\IC)$
is not isomorphic to $\lat(\mc{R})$, $\mc{R}$ 
 the hyperfinite von Neumann algebra factor of type II$_1$,
 as shown by von Neumann \cite{neu2}, cf.  the preface to \cite{neu}.
Compare this with \cite[p.453]{dunn}. 
Both ortholattices are simple and not in $\molf$,
in particular not subdirect products of
  $\lat(\IC^d)$'s (cf. \cite[p.449]{dunn}). $CG(\IC)$
admits a representation in an inner product space
over some ultrapower of $\IC$ (\cite{hs2})
but it remains open whether there is a representation within some
Hilbert space (cf. \cite[p.453]{dunn}).

Also, it remains an open problem
whether $\lat(R) \in \mc{N}$ (or at least in $\mc{V}(\molf)$)
for any ortholattice $L(R)$ of projections of a $*$-regular
ring $R$, where $R$ is $\IC$-algebra and the action of $\IC$ 
 compatible with the involutions; this is open even for
the case that $\lat(R)$ is continuous.

\section{Test sets} Cf. \cite[p.451, 454]{dunn}.
A subset $S$ of an MOL $L$ is a \emph{test set} for $L$
if any identity is valid in $L$ provided  it is so if
the variables are 
assigned to elements of $S$, only. It was shown in \cite[p.451]{dunn}        that 
the  $\lat(\IF^d)$,  $2\leq d<\infty$,  do not admit
finite test sets. The following extends this  result
to infinite simple $L\in\molf$. 
 \begin{pro}
For each $d,m>1$  there is an ortholattice identity
$\sigma_{d,m}(\bar z,\bar x)$ in $d+1+m$ variables such that,
for any infinite simple MOL $L$ of $d(L)=d$,
$\sigma_{d,m}$ fails in $L$ but is
 is satisfied under any assignment
identifying at least $2$ of the variables $x_i$.
\end{pro}
\begin{proof}
Fix $d$ and recall the terms $t^d_i=t^d_i(\bar z)$.
 Define \[s^d_j=s^d_j(\bar z,x_j)=
(t^d_0x_j)'(t^d_0+t^d_1)x_j +t^d_0t^d_1.\]
Now, consider an assignment $\bar z \mapsto \bar a$, $\bar x \mapsto \bar b$
in an MOL of $d(L)=d$.  Let $\hat{a}_i$ and $\hat{b}_j$
denote the values of the terms $t^d_i$ and $s^d_j$.
Then either all $\hat{a}_i, \hat{b}_j$ are equal or the following holds: the
$\hat{a}_i$ form 
a nontrivial $d$-diamond and  are atoms of $L$,
$\hat{b}_j$ is $0$ or an atom in
the $2$-dimensional interval $[0, \hat{a}_0+\hat{a}_1]$, and
$\hat{a}_0\hat{b}_j=0$.
Moreover, if the $a_i,b_j$ satisfy the relations stated for
the $\hat{a}_i, \hat{b}_j$ then $\hat{a}_i=a_i$ and $\hat{b}_j=b_j$
for all $i,j$. Thus,  
with the identity  $\sigma_{d,m}(\bar z,\bar x)$ given as
\[t_0(\bar z)t_1(\bar z)=t_0(\bar z) \prod_{j\neq k}(s_j(\bar z,x_j)+s_k(\bar z,x_k))\]
one has 
$\sigma_{d,m}$ satisfied by $\bar a,\bar b$
if $b_j=b_k$ for some $j\neq k$,
falsified if $\bar a$ is a nontrivial $d$-diamond and
$a_0\neq b_j\neq b_k$ for all $j\neq k$.
\end{proof}
Concerning $CG(\IC)$ and $\lat(\mc{R})$, a
 test set $S$ is provided by the union of any system of sub-ortholattices
$\lat((\IA \cap \IR)^d)$, $d \to \infty$. Such $S$
are given by  construction in case of $CG(\IC)$,
by \cite[Theorem XIV]{ropIV} in case of $\lat(\mc{R})$; 
  the system can be chosen so that elements of $S$ 
have rational normalized dimension
with denominators being  powers of $2$.

\section{Axiomatization} Cf. \cite[p.452]{dunn}.
Based  on orthonormal frames and 
 von Neumann coordinatization \cite{neu},
 for fixed $d\geq 3$, 
one derives an 
 axiomatization of 
  the first order theory of $\lat(\IF^d)$
from one of $\IF$ \cite{hzau}.
The first order theory of $\lat(\IF^d)$ 
is finitely axiomatizable if and only so is that of $\IF$;
in particular, finite axiomatizability fails for
$\IA \cap \IR \subseteq \IF \subseteq \IC$-

\section{Decision  problems} Cf. \cite[p.452-3]{dunn}.
 Roddy \cite{Roddy89} has constructed a simple MOL $L_{Rod}$  of height
$14$ which interprets an unsolvable word problem for division rings
and used this to   show that there is a finite  ortholattice presentation
which has unsolvable word problem in any variety
of MOLs which contains $L_{Rod}$. 
 The uniform word problem is unsolvable for any variety
of MOLs containing as subreducts
the subspace lattices of $F^d$, $d < \omega$, for a fixed prime field $F$
\cite{hn}.
In particular, the decision problem for  quasi-identities 
is unsolvable for $\{\lat(\IC^d)\mid d< \omega\}$.

Decidability of $\QLo(\mol)$
and $\QLo(\molf)$
 are open problems; the constructions yielding
unsolvability for varieties of modular lattices 
have no counterparts in $\mol$.
Using Roddy's result,
the equational theory of the variety of $d$-distributive 
MOLs has been shown undecidable for any  fixed
$d \geq 14$  \cite{Micol}.
For $L\in \molf$, 
according to \cite{hard}, $\QL(L)$ is decidable if and only if
the theory of quasi-identities of $L$ is decidable.

For a class $\mc{C}$ of MOLs, the
\emph{refutation problem} is the complement
of the decision problem for $\QLo(\mc{C})$;
that is, to decide for any given identity
whether is fails in some member of $\mc{C}$.
The \emph{satisfiability problem} for $\mc{C}$
is to decide for any given   equation
(equivalently, any conjunction of equations) whether
there is $L \in \mc{C}$ with $0\neq 1$
and a satisfying assignment in $L$.

For nontrivial  $L \in \molf$, these problems are p-time
equivalent to each other and NP-hard  \cite{LiCS,arxiv}.
For  $d(L)\leq 2$, they are  NP-complete.
For fixed $3\leq d <\omega$ and $L=\lat(\IF^d)$,
where $\IF \subseteq \IC$,
both problems are p-time equivalent 
to the problem FEAS$_{\IZ,\IF\cap\IR}$:
To decide for a finite list of  multivariate polynomials
with integer coefficients whether there is a common
zero in $\IF\cap \IR$. 
For $\IF\supseteq \IA\cap \IR$,
 the latter is 
 complete for the complexity  class  $\calBP(\calNP_{\IR}^0)$
in  Blum-Shub-Smale non-deterministic of real computation;
in particular, the problem is in 
$\calPSPACE$.

$\QLo(\mc{N})$ is decidable \cite{hard,hn}.
This follows from
decidability of the first order theory of each $\lat(\IC^d)$
and  the fact that an identity $\varepsilon$
 falsified
in some $\lat(\IF^n)$ is falsified in 
$\lat(\IF^{d(\varepsilon)})$  with computable
function $d$.
The function  $d$ can be chosen bounded by the
length of the expression $\varepsilon$.
This gives a p-time reduction of the refutation problem to  
 FEAS$_{\IZ,\IR}$;  p-time equivalence can be shown, too.
\xc{is shown
  in \cite{hvar}.}
The satisfiability problem is unsolvable for
any $\mc{C} \subseteq \mol$ such that,
for some $\IF\subseteq \IC$, $\{\lat(\IF^d)\mid d<\omega\}$
is contained  in the quasi-variety generated by $\mc{C}$; in  particular,
this applies to $CG(\IC)$ and $\lat(\mc{R})$.

For (additive) categories, enriched by adjunction,
of finite dimensional Hilbert spaces (cf. \cite{ab}),
the analogues of the mentioned hardness and undecidablity results
follow  interpretating. On the other hand,
given a bound on dimension, decidability and the complexity bound
follow by use of coordinates as in \cite{arxiv},
also  under further enrichment, e.g.  by tensor products.   If only 
adjunction  is added to the category
of all finite dimensional Hilbert spaces, 
decidability of equations 
can be obtained as follows. If an equation $\eta$ fails,
then it does so 
in a 
finite subcategory $\mc{C}$.
Form the orthogonal direct sum $H$ of objects in $\mc{C}$.
Then $\eta$ fails in ${\sf End}(H)$,
the endomorphism $*$-ring $R$  of $H$.
$R$ is isomorphic to 
 von Neumann coordinate $*$-ring of $\lat(H^3)$
and $\eta$ translates in an ortholattice  identity $\varepsilon$.
Use    $\frac{1}{3}d(\varepsilon)$ from above as a bound for $\dim H$
and the decision procedure for ${\sf End}(H)$.

\end{document}